\documentclass[a4paper,11pt]{amsart}
\usepackage{amsmath,amssymb,amsthm,mathtools}
\usepackage{hyperref,enumitem,xcolor}

\newtheorem{theorem}{Theorem}[section]
\newtheorem{lemma}[theorem]{Lemma}
\newtheorem{proposition}[theorem]{Proposition}

\theoremstyle{remark}
\newtheorem{remark}[theorem]{Remark}

\theoremstyle{definition}
\newtheorem{example}[theorem]{Example}

\title[Counterexample for the Daugavet index of thickness in $\ell_1$-sums]{A counterexample for the Daugavet index of thickness in $\ell_1$-sums}

\author{Rainis Haller}
\address{Institute of Mathematics and Statistics, University of Tartu, Estonia}
\email{rainis.haller@ut.ee}

\author{Andre Ostrak}
\address{Institute of Mathematics and Statistics, University of Tartu, Estonia}
\email{andre.ostrak@ut.ee}


\thanks{This work was supported by the Estonian Research Council grants (PRG1901) and (KOHTO32).}

\begin{document}

\begin{abstract}
We give a negative answer to a question of Haller--Lange\-mets--Lima--Nadel--Rueda Zoca asking whether, for all Banach spaces $X$ and $Y$, the Daugavet index of thickness satisfies
\[
T(X\oplus_1 Y)=\min\{T(X),T(Y)\}.
\]
We show that this equality does hold whenever one of the two summands has the Daugavet property. On the other hand, if $D$ is a Banach space with the Daugavet property and $N$ is a suitable absolute norm, then for $X=D\oplus_N D$, one has $T(X\oplus_1 X)<T(X)$.
\end{abstract}

\maketitle

\section{Introduction}
Let $X$ be a real Banach space. In \cite{MR3745574}, Rueda Zoca introduced the Daugavet index of thickness
\[
\begin{aligned}
T(X)=\inf\{r>0 :\ &\text{there are $x\in S_X$ and a non-empty relatively weakly}\\
&\text{open $W\subset B_X$ such that $W\subset B(x,r)$}\}.
\end{aligned}
\]
The index measures quantitatively how far $X$ is from having the Daugavet property. One has $0\leq T(X)\leq 2$ with $T(X)=2$ if and only if $X$ has the Daugavet property. For finite-dimensional $X$, one has $T(X)=0$. We shall use the standard geometric characterisation of the Daugavet property: if $X$ has the Daugavet property, then for every $x\in S_X$, every non-empty relatively weakly open subset $U$ of $B_X$, and every $\varepsilon>0$, there is $u\in U$ such that $\|x-u\|>2-\varepsilon$ (see \cite{MR1621757, MR1784413}).

In \cite{MR4211025}, the authors introduced the indices $T^s$ and $T^{cc}$ by replacing, in the definition of $T(X)$, relatively weakly open subsets with slices and with convex combinations of relatively weakly open subsets, respectively.
These indices satisfy
\[
0\leq T^{cc}(X)\leq T(X)\leq T^s(X)\leq 2
\]
and all three are equal to $2$ exactly in the Daugavet case. Their behaviour under $\ell_1$-sums was also studied in~\cite{MR4211025}. In particular, the estimate
\[
T(X\oplus_1 Y)\leq \min\{T(X),T(Y)\}
\]
from \cite{MR3745574} was recalled and the analogous inequality for $T^{cc}$ was proved. By contrast, for the slice version, one has the exact formula
\[
T^s(X\oplus_1 Y)= \min\{T^s(X),T^s(Y)\}.
\]
Whether the analogous equality holds for $T$ and for $T^{cc}$ was left open in \cite[Question~2.14]{MR4211025}.

The present note concerns the part of that question dealing with $T$. We show that the equality for $T$ fails even when $X=Y$. 
The construction remains within the framework of absolute sums used in \cite{MR4211025}. More precisely, let $0<s<t<1$ and define an absolute normalised norm $N$ on $\mathbb R^2$ by \[N(a,b)=\max\{ s|a|+|b|, |a|+t|b|\}.\]
For the standard facts on absolute normalised norms, see \cite{MR442682, MR1801250}.
If $D$ is a Banach space with the Daugavet property and $X=D\oplus_N D$, then 
\[
T(X)=1+t\quad\text{and}\quad T(X\oplus_1 X)=\frac{2(1-st)}{2-s-t},
\]
which gives the desired counterexample (see Theorem~\ref{thm:family-counterexample} below).

These indices have been further studied in \cite{RZ2025} and \cite{CJ2024}. In the latter paper, Choi and Jung considered the pointwise version of this quantitative theory, introducing the Daugavet constant $dc(x)$ and the super Daugavet constant  $sdc(x)$ of a point $x$.
In \cite[Proposition~2.14]{CJ2024}, they observed that 
\[
T^s(X)=\inf_{x\in S_X}dc(x)\quad\text{and}\quad T(X)=\inf_{x\in S_X}sdc(x).
\]
They also proved an $\ell_1$-sum lower estimate for the pointwise constant $dc$, which reflects the known stability of $T^s$ under $\ell_1$-sums at the pointwise level. In contrast, the example in the present note shows that the corresponding lower estimate fails for the super-Daugavet constant already at the level of the global infimum.

We consider only real Banach spaces. If $N$ is an absolute normalised norm on $\mathbb R^n$, we write 
\[
S_N^+=\{\alpha\in [0,\infty)^n : N(\alpha)=1\}.
\]

\section{A lower estimate for \texorpdfstring{$T(X\oplus_1 Y)$}{T(XO1Y)}}
For $\ell_1$-sums, the following estimate refines \cite[Proposition~2.2]{MR4211025}.
\begin{proposition}\label{prop:gen_lower_est}
    Let $X$ and $Y$ be Banach spaces with $T(X), T(Y)>0$.
    Then
    \[
    T(X\oplus_1 Y)\geq 
    \frac{2T(X)T(Y)}{2T(X)+2T(Y)-T(X)T(Y)}.
    \]
\end{proposition}

\noindent
Together with the known upper estimate, it implies in particular that
\[
T(X\oplus_1 Y)=\min\{T(X),T(Y)\}
\]
whenever one of $X$ and $Y$ has the Daugavet property. 

\begin{remark}
    It is sometimes useful to rewrite Proposition~\ref{prop:gen_lower_est} in the following form. For a Banach space $Z$ with $T(Z)>0$, put $d(Z)=\frac{2}{T(Z)}-1$. In this terminology, $d(Z)=0$ precisely when $T(Z)=2$, that is, precisely when $Z$  has the Daugavet property. Proposition~\ref{prop:gen_lower_est} says that $d(X\oplus_1 Y)\leq d(X)+d(Y)$. On the other hand, the known estimate $T(X\oplus_1 Y)\leq \min\{T(X),T(Y)\}$ is equivalent to $\max\{d(X),d(Y)\}\leq d(X\oplus_1 Y)$. Thus
    \[
    \max\{d(X),d(Y)\}\leq d(X\oplus_1 Y)\leq d(X)+d(Y).
    \]
\end{remark}

To prove the proposition, we make use of the following standard fact.
\begin{lemma}\label{lem:W_S}
    Let $X$ be an infinite-dimensional Banach space, $x\in S_X$, and $W$ a non-empty relatively weakly open subset of $B_X$. Then 
    \[
    \sup_{w\in W}\|x-w\|=\sup_{w\in W\cap S_X}\|x-w\|.
    \]
\end{lemma}

\begin{proof}
    Only the inequality ``$\leq$'' needs proof. Fix $w\in W$ and $\varepsilon>0$. 
    Since the function $X\ni u\mapsto \|x-u\|\in \mathbb{R}$ is weakly lower semicontinuous, the set
    \[
    U=\{u\in B_X\colon \|x-u\|>\|x-w\|- \varepsilon\}
    \]
    is relatively weakly open, and so $W\cap U$ is a relatively weakly open neighbourhood of $w$. Therefore, $W\cap U\cap S_X$ is non-empty, and so
    \[
    \sup_{u\in W\cap S_X}\|x-u\|\geq \|x-w\|-\varepsilon.\qedhere
    \]
\end{proof}

\begin{lemma}\label{lem:gen_ab_geq}
    Let $X$ be a Banach space and let $x,y\in S_X$. If $\|x-y\|\geq C-\varepsilon$, where $0\leq C\leq 2$ and $\varepsilon>0$, then for $a,b\in [0,1]$,
    \[
    \|ax-by\|\geq
    \max\big\{|a-b|,\max\{a,b\} C-|a-b|\big\}-\varepsilon.
    \]
\end{lemma}

\begin{proof}
   The inequality $\|ax-by\|\geq |a-b|$ is immediate. It remains to note that
   \[
   \|ax-by\|\geq \max\{a,b\}\|x-y\|-|a-b|\geq \max\{a,b\}(C-\varepsilon)-|a-b|.
   \]
\end{proof}

\begin{proof}[Proof of Proposition~\ref{prop:gen_lower_est}]
     
     Let $Z=X\oplus_1 Y$ and let $W$ be a non-empty relatively weakly open subset of $B_Z$. Let $z=(ax,by)\in S_Z$ and $w_0=(cu_0,dv_0)\in W\cap S_Z$, where $x,u_0\in S_X,y,v_0\in S_Y$ and $a,b,c,d\in [0,1]$, $a+b=c+d=1$. 
    Choose non-empty relatively weakly open subsets $U$ of $B_X$ and $V$ of $B_Y$ such that 
    \[
    (cU)\times (dV)\subset W.
    \] 
    By the definition of $T(X)$, Lemma~\ref{lem:W_S}, and Lemma~\ref{lem:gen_ab_geq}, we may choose     
    $u\in U\cap S_X$ so that, up to an arbitrarily small error,
    \[
    \|ax-cu\|\geq \max\{|a-c|,\max\{a,c\}T(X)-|a-c|\}\eqqcolon \alpha.
    \]
    Similarly, up to an arbitrarily small error, for some $v\in V\cap S_Y$,
    \[
    \|by-dv\|\geq \max\{|b-d|,\max\{b,d\}T(Y)-|b-d|\}\eqqcolon \beta.
    \]
    Hence,
    \[
    \sup_{w\in W}\|z-w\|\geq \alpha+\beta.
    \]
    It suffices to show that $\alpha+\beta\geq 2AB/C$, where $A=T(X)$, $B=T(Y)$, and $C=2A+2B-AB$.
    
    We look at the case $a\geq c$; the case $a<c$ is analogous. Denote $\theta=a-c$ and note that $d-b=\theta$ and $a+d=1+\theta$. 
    Then
    \[
    \alpha=\max\{\theta, aA-\theta\}\geq (1-\frac{2B}{C})\theta+\frac{2B}{C}(aA-\theta)=(1-\frac{4B}{C})\theta+\frac{2AB}{C}a
    \]
    and
    \[
    \beta=\max\{\theta,dB-\theta\}\geq (1-\frac{2A}{C})\theta+\frac{2A}{C}(dB-\theta)=(1-\frac{4A}{C})\theta+\frac{2AB}{C}d.
    \]
    Therefore,
    \begin{align*}
         \alpha+\beta&\geq \big(2-\frac{4A+4B}{C}\big)\theta+\frac{2AB}{C}(a+d)\\
         &= \big(2-\frac{4A+4B-2AB}{C}\big)\theta+ \frac{2AB}{C}=\frac{2AB}{C}.\qedhere
    \end{align*}
\end{proof}

\begin{remark}
When applied to the $\ell_1$-norm, \cite[Proposition 2.2]{MR4211025} gives
\[
T(X\oplus_1Y)\ge 2\bigl(\min\{T(X),T(Y)\}-1\bigr).
\]
Proposition~\ref{prop:gen_lower_est} gives a stronger bound. In particular, our estimate gives
\[
T(X\oplus_1 X)\ge\frac{2T(X)}{4-T(X)}.
\]
\end{remark}

\section{Absolute sums of Daugavet spaces and the counterexample}

Let $N$ be an absolute normalised norm on $\mathbb R^n$ ($n\geq 2$) and set
\[
\mu(N)=\inf\{N(\alpha+\beta):\alpha,\beta\in S_N^+\}.
\]
\begin{remark}\label{rem:n=2_m(N)=N(1,1)}
If $n=2$, then $\mu(N)=N(1,1)$.
\end{remark}

\begin{proposition}\label{prop:daugavet-absolute-lower}
Let $D$ be a Daugavet space and let $X=D\oplus_N\dotsb\oplus_N D.$ 
Then
\[
T(X)\ge \mu(N).
\]
\end{proposition}

\begin{proof}
Let $U$ be a non-empty relatively weakly open subset of $B_X$, let $x=(\alpha_1 x_1,\dotsc,\alpha_n x_n)\in S_X$, and let $u=(\beta_1 u_1,\dotsc,\beta_n u_n)\in U\cap S_X$,
where $x_1,u_1,\dotsc,x_n,u_n\in S_D$, $\alpha\coloneqq (\alpha_1,\dotsc,\alpha_n)\in S_N^+$, and $\beta\coloneqq (\beta_1,\dotsc,\beta_n)\in S_N^+$.

Choose non-empty relatively weakly open subsets $U_i$ of $B_{D}$ such that
\[
\beta_1U_1\times\dotsb\times \beta _nU_n\subset U.
\]

Fix $\varepsilon>0$. Use the Daugavet property together with Lemma~\ref{lem:W_S} to choose $w_i\in U_i\cap S_{D}$ such that
\[
\left\|x_i-w_i\right\|>2-\varepsilon.
\]
By Lemma~\ref{lem:gen_ab_geq},
\[
\|\alpha_i x_i-\beta_i w_i\|\ge \alpha_i+\beta_i-\varepsilon.
\]

Let $w=(\beta_1w_1,\dotsc,\beta_n w_n)\in U$. By monotonicity of $N$, the coordinate estimates give
\[
\|x-w\|\ge N(\alpha+\beta)-n\varepsilon.
\]
Hence,
\[
\sup_{w\in U}\|x-w\|\geq N(\alpha+\beta)\geq \mu(N).\qedhere
\]
\end{proof}

\begin{remark}
    The same proof works for $D_1\oplus_N\dotsb\oplus_N D_n$, where $D_1,\dotsc,D_n$ are Daugavet spaces.
\end{remark}

\begin{theorem}\label{thm:family-counterexample}
Let $0<s<t<1$ and let $N$ be an absolute normalised norm on $\mathbb{R}^2$ defined by
\[N(a,b)=\max\{s|a|+|b|,\ |a|+t|b|\}.\]
Let $X=D\oplus_N D$, where $D$ is a Banach space with the Daugavet property. Then $T(X)=1+t$ and
\[T(X\oplus_1X)=\frac{2(1-st)}{2-s-t}.\]
In particular, $T(X\oplus_1X)<T(X)$.
\end{theorem}

\begin{proof}
    We start by proving $T(X)=1+t$.
    The lower estimate follows from Remark~\ref{rem:n=2_m(N)=N(1,1)} and Proposition~\ref{prop:daugavet-absolute-lower}, since $N(1,1)=1+t$.

For the reverse inequality, let $x\in S_D$ and let $\varepsilon>0$. Then $(x,0)\in S_{X}$.
We define the relatively weakly open subset of $B_X$,
\[
W=\{(u,v)\in B_{X}: \|v\|>1-\varepsilon\}.
\]
Fix $(u,v)\in W$. Since $\|v\|> 1-\varepsilon$ and $s\|u\|+\|v\|\le 1$, we get $\|u\|<\varepsilon/s$.
Thus,
\[
\|(x,0)-(u,v)\|
=N(\|x-u\|,\|v\|)
\le N(1+\varepsilon/s,1).
\]
Letting $\varepsilon\to 0$ gives $T(X)\le N(1,1)=1+t$.

We now prove $T(X\oplus_1X)=\frac{2(1-st)}{2-s-t}$. Denote $\gamma=\frac{2(1-st)}{2-s-t}$ and put
\[
\delta=\min\{N(a,b):a,b\ge0,\ a+b=1\}.
\]
If $a,b\geq 0$ and $a+b=1$, then $N(a,b)=\max\{sa+1-a, a+t(1-a)\}$. The minimum thus is obtained when $sa+1-a=a+t(1-a)$, i.e. at $a=\frac{1-t}{2-s-t}$ and $b=\frac{1-s}{2-s-t}$, which gives that $\delta=(1-st)/(2-s-t)=\gamma/2$.

Let $M$ be the absolute norm on $\mathbb R^4$ defined by
\[
M(a,b,c,d)
=N(a,b)+N(c,d).
\]
Then $X\oplus_1X$ is the $M$-sum of four copies of $D$. Thus, by Proposition~\ref{prop:daugavet-absolute-lower}, $T(X\oplus_1 X)\geq \mu(M)$. To see that $\mu(M)\geq \gamma$, take $\alpha=(\alpha_1,\dotsc,\alpha_4),\beta=(\beta_1,\dotsc,\beta_4)\in S_{M}^+$, and note that
\[
\begin{aligned}
M(\alpha+\beta)
&=N(\alpha_1+\beta_1,\alpha_2+\beta_2)
 +N(\alpha_3+\beta_3,\alpha_4+\beta_4)\\
&\ge \delta(\alpha_1+\beta_1+\alpha_2+\beta_2)+\delta(\alpha_3+\beta_3+\alpha_4+\beta_4)\\
&=\delta(\alpha_1+\alpha_2+\alpha_3+\alpha_4)+\delta(\beta_1+\beta_2+\beta_3+\beta_4)\\
&\geq \delta+\delta=\gamma.
\end{aligned}
\]

For the inequality $T(X\oplus_1 X)\leq \gamma$, write 
$Z=X\oplus_1X$ and choose $x\in S_D$. Let $a=\frac{1-t}{2-s-t}$ and $b=\frac{1-s}{2-s-t}$.
Then 
\[
N(a,b)=\frac{1-st}{2-s-t}\quad \text{and}\quad\|((ax,0),(0,bx))\|=a+b=1.
\]

Let $\varepsilon>0$ and define
\[W=\{((u_1,u_2),(u_3,u_4))\in B_Z : \|u_2\|>b(1-\varepsilon),\ \|u_3\|>a(1-\varepsilon)\}.\]
This is a relatively weakly open subset of $B_Z$ that contains $((0,bx),(ax,0))$.
Let $((u_1,u_2),(u_3,u_4))\in W$. Then
\[ 
    \|u_2\|\ge b(1-\varepsilon)\quad\text{and}\quad \|u_3\|\ge a(1-\varepsilon).
\]
Since
\[
s\|u_1\|+\|u_2\|\leq N(\|u_1\|,\|u_2\|)\leq 1-N(\|u_3\|,\|u_4\|)\leq 1-\|u_3\|\leq b+a\varepsilon,
\]
one obtains $\|u_1\|\leq \varepsilon/s$ and $\|u_2\|\leq b+a\varepsilon$.
Thus,
\[
\|(ax,0)-(u_1,u_2)\|
\le
N(a+\varepsilon/s 
,\ b+a\varepsilon).
\]
Similarly, $\|u_3\|\leq a+b\varepsilon$ and $\|u_4\|\leq \varepsilon/t$,
and hence,
\[\|(0,bx)-(u_3,u_4)\|\le N(a+b\varepsilon, b+{\varepsilon}/{t}).
\]
Letting $\varepsilon\to 0$, we get $T(X\oplus_1X)\le2 N(a,b)=\gamma$.

The strict inequality $T(X\oplus_1X)<T(X)$ follows from
\[
1+t-\frac{2(1-st)}{2-s-t}
=\frac{(1-t)(t-s)}{2-s-t}>0.\qedhere
\]
\end{proof}

\begin{remark}
In Theorem~\ref{thm:family-counterexample},
\[T(X)-T(X\oplus_1X)=\frac{(1-t)(t-s)}{2-s-t}.\]
For fixed $t$, this quantity increases as $s\downarrow0$, and the limit is $(1-t)t/(2-t)$, whose maximum over $0<t<1$ is $3-2\sqrt2$, attained at $t=2-\sqrt2$.
\end{remark}

The following example shows that the general lower bound for $T(X\oplus_1 Y)$ in terms of the two numbers $T(X)$ and $T(Y)$ obtained in Proposition~\ref{prop:gen_lower_est} is not always sharp. 
\begin{example}\label{ex:concrete}
Let $N$ be an absolute normalised norm on $\mathbb{R}^2$ defined by
\[
N(a,b)=\max\{\tfrac1{10}|a|+|b|,\ |a|+\tfrac12|b|\}
\]
and let $X=C[0,1]\oplus_N C[0,1]$.
Then $T(X)=3/2$ and $T(X\oplus_1X)=19/14$,
whereas Proposition~\ref{prop:gen_lower_est} gives only
\[T(X\oplus_1X)\ge\frac{2T(X)}{4-T(X)}=\frac65.\]
\end{example}

\begin{remark}
Theorem~\ref{thm:family-counterexample} answers \cite[Question 2.14 (a)]{MR4211025} negatively. The same examples do not answer Question~2.14 (b), concerning $T^{cc}$. In fact, with the notation of Theorem~\ref{thm:family-counterexample}, it is possible to show that
\[
T^{cc}(X)=T^{cc}(X\oplus_1 X)=\frac{2(1-st)}{2-s-t}.
\]
Indeed, the inequality $T^{cc}(X)\geq T^{cc}(X\oplus_1 X)$ holds by~\cite{MR4211025}. Proposition~\ref{prop:daugavet-absolute-lower} has the following convex-combination version: if
\[
\mu^{cc}(N)=\inf\{N(\alpha+\beta):\alpha\in S_N^+, \beta\in \operatorname{conv}(S_N^+)\},
\]
then
\[
T^{cc}(D\oplus_N\dotsb\oplus_N D)\geq \mu^{cc}(N).
\]
Here, one uses the standard convex-combination form of the Daugavet property. For the norm $M=N\oplus_1 N$, one has $\mu^{cc}(M)=\frac{2(1-st)}{2-s-t}$, which gives a lower estimate $T^{cc}(X\oplus_1 X)\geq \frac{2(1-st)}{2-s-t}$. The estimate $T^{cc}(X)\leq \frac{2(1-st)}{2-s-t}$ follows by taking $x\in S_D$, $z=(0,x)$, and the convex combination
\[
C\coloneqq\lambda\{(u,v)\in B_X : \|u\|> 1-\varepsilon\}+(1-\lambda)\{(u,v)\in B_X : \|v\|> 1-\varepsilon\},
\]
where $\lambda=2(1-t)/(2-s-t)$.
Then every $w\in C$ satisfies
\[
\|z-w\|\leq
N(\lambda+(1-\lambda)\varepsilon/s,\,
2-\lambda+\lambda\varepsilon/t),
\]
and, letting $\varepsilon\to 0$, the right-hand side tends to
\[
N(\lambda,2-\lambda)=\frac{2(1-st)}{2-s-t}.
\]
Thus, \(T^{cc}(X)\leq 2(1-st)/(2-s-t)\), and the equality follows.
\end{remark}

\begin{remark}
The construction in Theorem~\ref{thm:family-counterexample} also gives a counterexample for the weak-star version, where $T_{w^*}$ is defined using relatively weak-star open subsets of the dual unit ball. Let $D=L_\infty[0,1]=L_1[0,1]^*$, which has the Daugavet property, and let $X=D\oplus_ND$ with 
\[N(a,b)=\max\{s|a|+|b|, |a|+t|b|\}.\]
Then $X$ is the dual of $L_1[0,1]\oplus_{N^*}L_1[0,1]$, and $X\oplus_1 X$ is also a dual space with the corresponding canonical predual. Since $T(Z)\leq T_{w^*}(Z)$ for every dual space $Z$, Theorem~\ref{thm:family-counterexample} gives
\[
T_{w^*}(X)\geq 1+t
\quad\text{and}\quad
T_{w^*}(X\oplus_1X)\geq\frac{2(1-st)}{2-s-t}.
\]
For the reverse inequalities, note that the relatively weakly open sets used in the upper estimates in the proof of Theorem~\ref{thm:family-counterexample} are, in the present dual setting, relatively weak-star open. Indeed, the coordinate projections are weak-star continuous and the norm on $L_\infty[0,1]$ is weak-star lower semicontinuous. Hence, the same upper estimates give
\[
T_{w^*}(X)\leq 1+t
\quad\text{and}\quad
T_{w^*}(X\oplus_1X)\leq\frac{2(1-st)}{2-s-t}.
\]
Therefore,
\[
T_{w^*}(X)=1+t
\quad\text{and}\quad
T_{w^*}(X\oplus_1X)=\frac{2(1-st)}{2-s-t}.
\]
Thus, the corresponding $\ell_1$-sum equality for $T_{w^*}$ fails as well.
\end{remark}

\section*{Acknowledgements}
This work was supported by the Estonian Research Council grants (PRG1901) and (KOHTO32).

\bibliographystyle{amsplain}
\bibliography{bibliography}

\end{document}